\documentclass[runningheads]{llncs}
\input{001_setup}

\begin{document}
\title{Exploring the Steiner-Soddy Porism}
\author{Ronaldo Garcia\inst{1}\orcidID{0000-0001-8876-6956} \and
Liliana Gheorghe\inst{2}\orcidID{0000-0003-3922-0177} \and
Dan Reznik\inst{3}\orcidID{0000-0002-0542-6014}}

\authorrunning{R. Garcia et al.}
\institute{Federal Univ. of Goiás, Goiânia, Brazil\\
\email{ragarcia@ufg.br}\\
\and
Federal Univ. of Pernambuco, Recife, Brazil\\
\email{liliana@dmat.ufpe.br}\\
\and
Data Science Consulting, Rio de Janeiro, Brazil\\
\email{dreznik@gmail.com}}
\maketitle              
\vspace{-0.75cm}
\begin{abstract}
We explore properties and loci of a Poncelet family of polygons -- called here Steiner-Soddy -- whose vertices are centers of circles in the Steiner porism, including conserved quantities, loci, and its relationship to other Poncelet families.
\end{abstract}
\vspace{-0.75cm}

\section{Introduction}
A {\em Steiner chain} is a set of pairwise-tangent circles, all of whom are tangent to a pair of disjoint circles, called the inner and outer ``Soddy'' circles, see \cref{fig:intro-n5}(left). The chain is poristic since it is the inversive image of a set of identical, mutually-tangent circles, centered at the vertices of a regular polygon, see \cref{fig:intro-n5}(right).
In this article, we explore the family of ``Steiner-Soddy'' polygons whose vertices are the centers of circles in a Steiner porism. 
\vspace{-0.33cm}
\subsubsection*{Main Results.}
Let $\P$ denote polygons in the Steiner-Soddy family. Let $N$ denote the number of its polygonal sides.

\begin{compactitem}
    \item The $\P$ are conic-inscribed and circumscribe a circle, and are therefore a Poncelet porism \cite{dragovic11}.
    \item The outer conic of the $\P$ has foci on the centers of the inner and Soddy circles of the Steiner chain.
    \item When the center of the caustic is on the circumference of the inner Soddy circle, the family becomes parabola-inscribed and the outer Soddy circle degenerates to a line.
    \item As a corollary to a result proved in \cite{schwartz2020-steiner}, the family conserves the sum of powers of half-angle tangents, up to power $N-1$.
    \item The locus of the perimeter centroid\footnote{This is the weighted average of the midpoints of sides, where the weights are sidelengths.} is a conic. Recall this is not guaranteed for a generic Poncelet family \cite{sergei2016-com}.
    \end{compactitem}
    
For the case when the $\P$ are triangles ($N=3$) we obtain:

    \begin{compactitem}
        \item The half-tangents of the family are equal to the cotangents of its intouch triangles.
    \item The sum of tangents of half angles is less than, equal, or greater than two if the family is ellipse, parabola, or hyperbola inscribed, respectively.
    \item When the family parabola-inscribed, the locus of the orthocenter is a line.
    \item In the spirit of \cite{odehnal2011-poristic}, we tour loci of some triangle centers over the $\P$, showing some to be stationary, circles, lines, conics, and non-conics.
\end{compactitem}

\begin{figure}
    \centering
    \includegraphics[width=\textwidth]{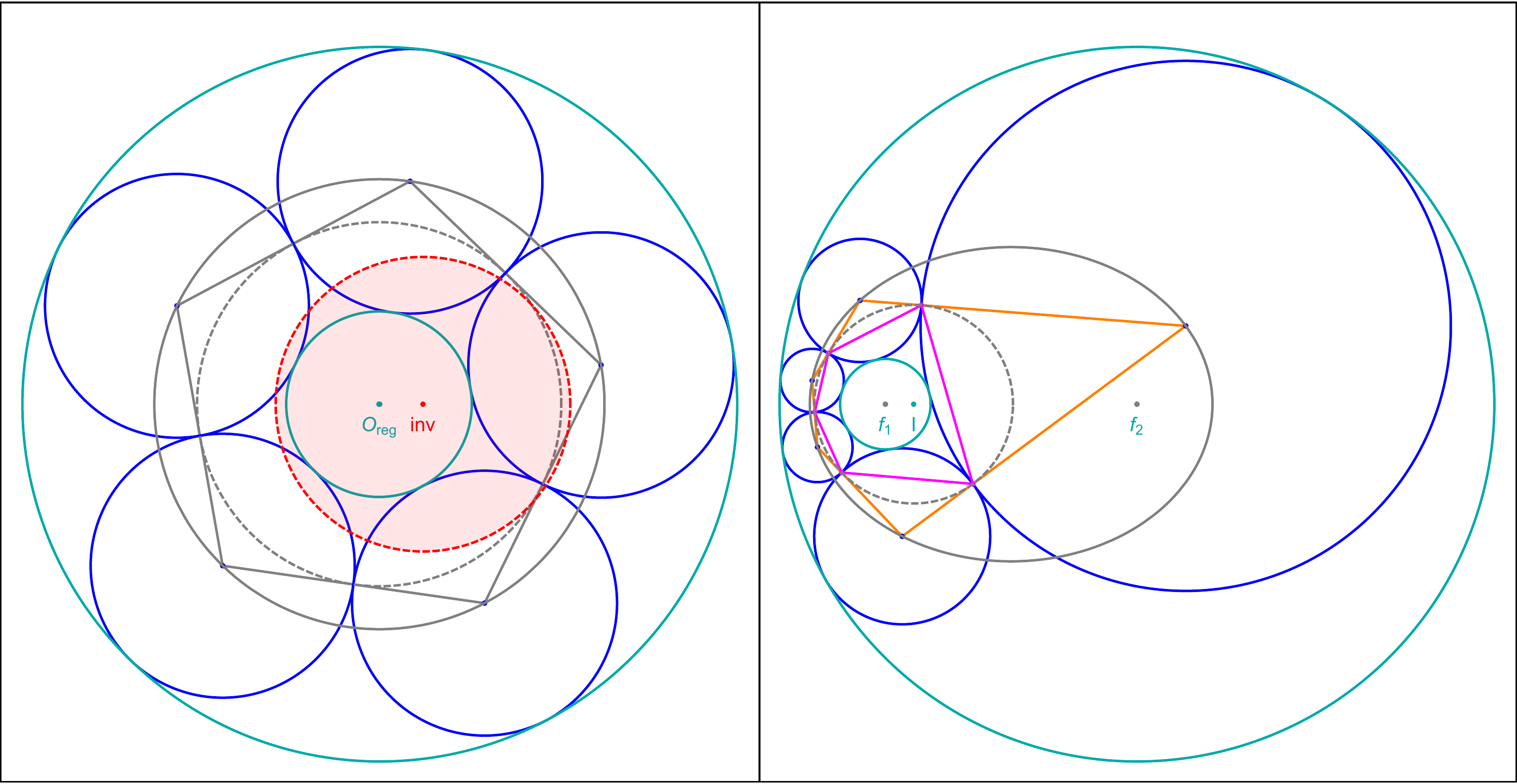}
    \caption{\textbf{Left:} Consider a set of pairwise tangent circles (blue) centered on the vertices of a regular polygon (gray). By symmetry, an inner and an outer ``Soddy'' circle (light blue) can be drawn tangent to all circles. Let the shaded red circle be an inversion circle. 
    \textbf{Right:} a Steiner chain is the inversive image of the setup in the left with respect to said inversion circle. Also shown is the Steiner-Soddy polygon $\P$ (orange) whose vertices are the centers of circles in the chain. A Poncelet porism of $\P$ exists inscribed in a conic (gray) whose foci are the centers of the inversive images of the two original Soddy circles (light blue). The caustic is a circle (dashed gray) centered at $I$. Also shown is the pedal polygon of $\P$ (magenta) with respect to $I$, also Ponceletian.}
    \label{fig:intro-n5}
\end{figure}

\vspace{-0.33cm}
\subsubsection*{Related Work.}
In \cite{schwartz2020-steiner} it is proved that the first $k-1$ moments of curvatures are invariant over Steiner's porism. Loci of vertex, area, and perimeter centroids of a generic Poncelet family are studied in \cite{sergei2016-com}, where it is shown that the first two are always conics. In \cite{yff2007-isoperim} a condition is derived which guarantees that a certain center of a triangle exists iff the sum of half-tangents is less than two; this is equivalent to whether the center of the outer Soddy circle is everted or not \cite{garcia2021-triads}. The Poncelet family which is the polar image of bicentrics was studied in \cite{bellio2021-parabola-inscribed}, and that of harmonic polygons was studied in \cite{roitman2021-bicentric}. Seminal studies of loci of triangle centers over families of Poncelet triangles include \cite{odehnal2011-poristic,skutin2013,zaslavsky2001-poncelet,zaslavsky2003-trajectories}. A theory for the type of locus swept by triangle centers over the confocal Poncelet family is presented in \cite{helman2021-theory}.
\vspace{-0.33cm}
\subsubsection*{Article organization.}
In \cref{sec:all-n} we prove the main properties of the Steiner-Soddy family for all $N$. In \cref{sec:n3} we prove certain results specialized to the $N=3$ case. Loci of triangle centers in the $N=3$ are toured \cref{sec:n3-loci}. We finish with a discussion in \cref{sec:cons} comparing invariants of various Poncelet families. Explicit formulas for some of the objects mentioned herein appear in \cref{app:explicit}.

\section{The Steiner-Soddy Porism}
\label{sec:all-n}
Consider polygons $\P$ with vertices at the centers of circles in a Steiner porism. We call these ``Steiner-Soddy'' polygons (or family). Referring to \cref{fig:intro-n5}, the points of contact between consecutive circles in the chain are concyclic, since these lie on the inversive image of a regular polygons' incircle. Call this circle $\C$ and its center $I$. Let $\H$ denote the circle-inscribed polygon whose vertices are said contact points. $\H$ is known to be {\em harmonic}, i.e., it is circle-inscribed and contains a special point $K$ whose distance to the sides are a constant proportion of the sidelengths \cite{casey1888}. 

It is known that the tangent lines between consecutive pairs of circles in a Steiner chain meet at $I$ \cite[pp. 120,244--245]{wells1991}. Therefore:

\begin{remark}
$\P$ is the pedal polygon of $\H$ with respect to $I$, and therefore the sides of $\P$ are tangent to $\C$.
\end{remark}


The Steiner porism has two fixed ``Soddy'' circles $\S$ and $\S'$ which are tangent to all $N$ circles in chain. Let $\S_{reg}$ and $\S'_{reg}$ denote their inversive pre-images in the regular setup,\cref{fig:intro-n5}(left).

\begin{proposition}
$\P$ is inscribed in a conic $\E$ which is an (i) ellipse, (ii) hyperbola, or (iii) parabola if the inversion center is (i) interior to $\S_{reg}$ or exterior to $\S'_{reg}$, (ii) in the annulus between) $\S_{reg}$ and $\S'_{reg}$, or (iii) on either circle. 

Furthermore, the foci of $\E$ are the centers of the inner and outer Soddy circles of $\S$ and $\S'$. In the case of a parabola, the center of $\S'$ is at infinity and $\S'$ becomes a line parallel to the directrix.
\label{prop:ssf}
\end{proposition}

\begin{proof}
It is well-known that locus of the center of a circle internally tangent to two fixed circles pair of circles is an ellipse (resp. hyperbola) if one circle is contained in another (resp. disjoint). Likewise, said locus when one circle has infinite radius is a parabola \cite{lucca2009-chains}.
\end{proof}

Since $\P$ is inscribed to a conic $\E$ and circumscribed a circle $\C$:

\begin{corollary}
The family $\P$ is Ponceletian.
\end{corollary}
In \cite{roitman2021-harmonic} it is shown that the pedal family $\H$ is also Ponceletian. Its envelope or caustic is known as the Brocard inellipse $\B$, see \cref{app:explicit} for explicit expressions.

Let $\theta_i$ denote the internal angle of $\P$ at its $i$-th vertex $P_i$. 
\begin{theorem}
The $\P$ family conserves  $\sum_{i=1}^N\tan(\theta_i/2)$ is conserved. 
\label{thm:sum_tang_n}
\end{theorem}

\begin{proof} Referring to
\cref{fig:intro-n5}, let $P_i$ (resp. $P_i'$) be the centers (resp. the contact points between) consecutive circles in the Steiner chain. Let $r$ (resp. $r_i$) denote the radius of $\C$ (resp. a circle in chain centered at $P_i$). Since $|P_i P_{i}'|=|P_i P_{i+1}'|=r_i$ and 
$|OP_{i}'|=|OP_{i+1}'|=r$, the line $P_i O$  bisects  $\widehat{{P_i'}P_iP_{i+1}'}=\theta_i$. Since $\H$ is the $I$-pedal of $P$, $\C$ is tangent to $P_iP_{i+1}$ at $P_i'$, so
 $\triangle{P_i P'_i O}$ is right-angled, i.e., $\widehat{P_i P_{i} O}=90^{\circ}$. Hence,
$\tan( \widehat{P_i P_{i} O})=\frac{O P_i'}{P_i P_i'}$. Thus:
\begin{equation}
\tan\frac{\theta_i}{2}=
\frac{r}{r_i}\;\;\;\mbox{hence}\;\;\;
\sum_{n=1}^{n}\tan^k(\theta_i/2)=
\sum_{n=1}^{n} \left(\frac{r}{r_i} \right)^k,\;\; \mbox{for any $k$.}
\label{eq:half-tan}
\end{equation}
The claim is obtained by invokig a result from \cite[Theorem 1]{schwartz2020-steiner}, namely that over Steiner's porism, the sum of the $k^{th}$ powers of curvatures of circles in chain, up to $N-1$, is invariant.
\end{proof}


\begin{remark}
When $\E$ is a hyperbola, the sign of $\tan(\theta_i/2)$ must be flipped for the two angles whose neighboring vertices lie on different branches of the hyperbola.
\end{remark}

\begin{remark} Richard Schwartz \cite{schwartz2021-private} has suggested an elegant interpretation of conservations of the above type:
All homogeneous polynomials in 2-variables of degree less than $N$ have the same average on the unit circle as they do on a regular $N$-gon inscribed in the circle.
\end{remark}

\begin{figure}
    \centering
    \includegraphics[width=\textwidth]{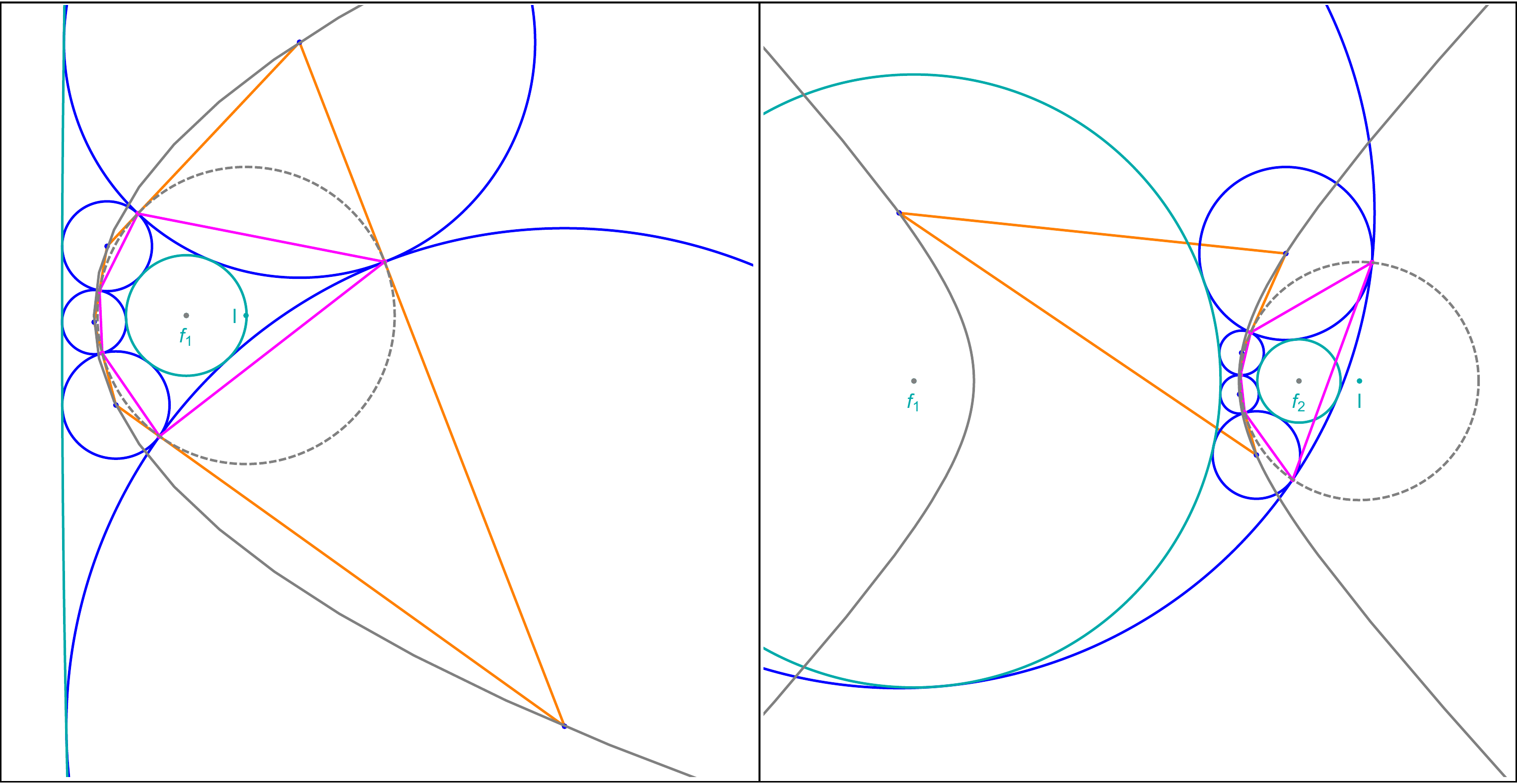}
    \caption{\textbf{Left:} If the inversion center is on $\S_{reg}'$, the Soddy polygon $P$ (orange) is inscribed to a parabola (gray), and $\S'$ degenerates to a line (vertical cyan), parallel to the directrix; in this case, $I$, the center of the caustic ${\C}$ (dashed grey), is on $\S$ (cyan circle).
    \textbf{Right:} When $\S$ and $\S'$ are disjoint, the family is hyperbola-inscribed. The family oscillates between two states: (i) as shown, all but one vertex lie on a single branch, and (ii) not shown, all vertices lie on a single branch.}
    \label{fig:n5-deg-ev}
\end{figure}
\begin{proposition}
Over a hyperbola-inscribed Steiner-Soddy family, the vertices of $\P$ oscillate between two states: (i) all on a single branch of the hyperbola or (ii) $N-1$ contiguous vertices on one branch and one on the other branch. This corresponds to whether $I$ is interior (resp. exterior) to $\H$, the $I$-pedal polygon.
\end{proposition}

\begin{proof}
When $\S$ and $\S'$ are disjoint (hyperbola case), the circles in the chain are either (all) externally tangent to both, or there exists a circle $\Gamma$ which is internally tangent to one of the  Soddy circles. In this case, $\Gamma$ is centered at a point which lies on the distal branch of the hyperbola, and is also the unique solution to an Apollonius problem for the  two Soddy circles and one neighbor in the chain, which is  externally tangent to all of them.
So there can be at most one center on the distal branch. In this case, $\Gamma$ is internally tangent to  two of the circles in the chain and the perpendiculars at the tangency points meet at $I$, which lies outside the pedal polygon.
\end{proof}

\vspace{-0.33cm}
\subsubsection*{Relationship with the Homothetic family} Let the {\em homothetic family} refer to a Poncelet family interscribed between two concentric, homothetic conics. Let their foci be called ``inner'' and ``outer'' ones. In \cite{roitman2021-bicentric} it was shown that the harmonic family is the polar image of the homothetic family with respect to an inversion circle centered on one of the outer foci. Referring to \cref{fig:homoth}:

\begin{proposition}
The Steiner-Soddy family is the polar image of the homothetic family with respect to an inversion circle centered on one of its inner foci.
\end{proposition}

\begin{figure}
    \centering
    \includegraphics[width=.6\textwidth,frame]{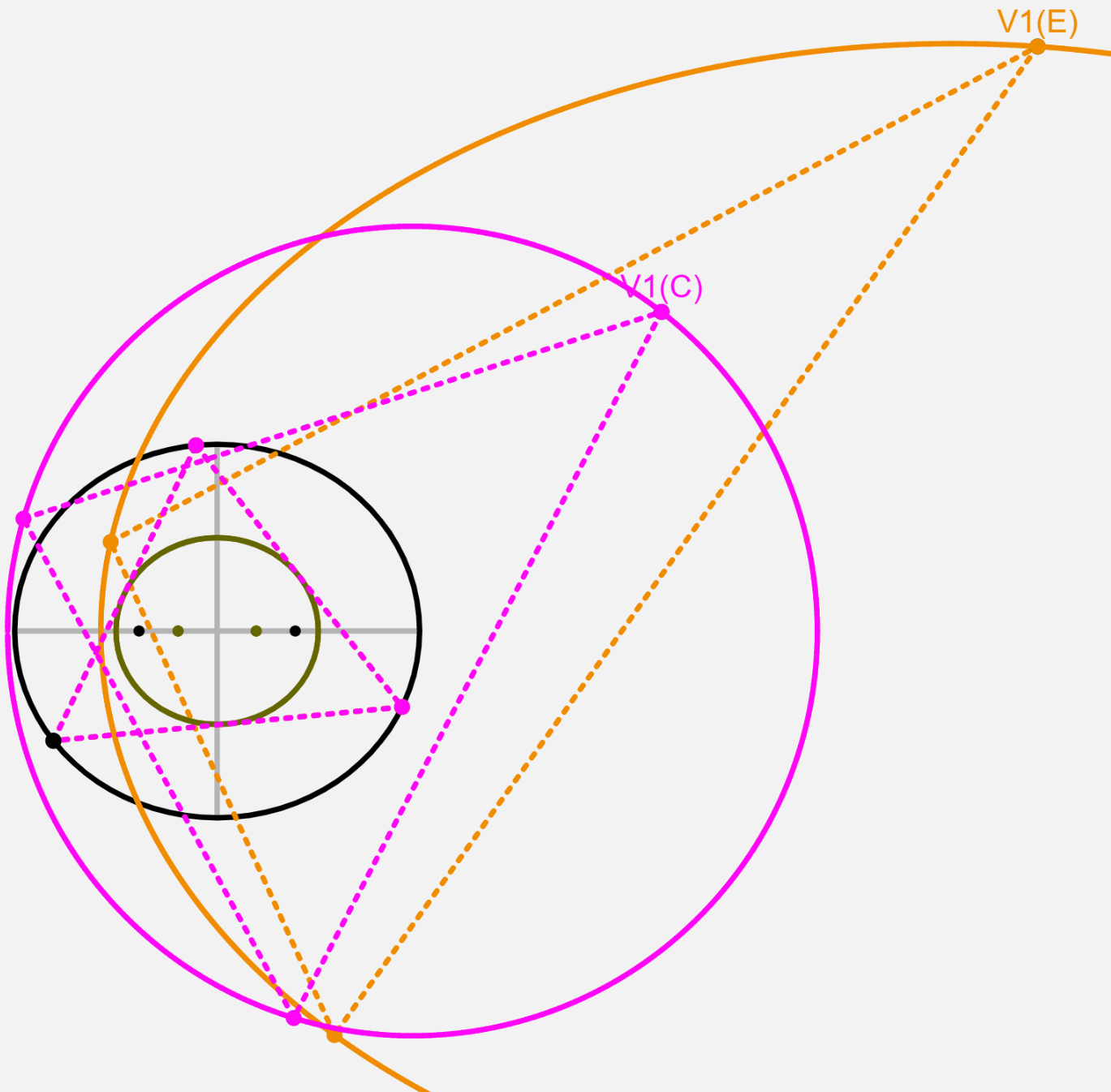}
    \caption{The Steiner-Soddy family, orange (resp. harmonic family, magenta), is the polar image of the ``homothetic'' family, with respect to a circle centered on an inner (resp. outer) focus.}
    \label{fig:homoth}
\end{figure}

\vspace{-0.33cm}
\subsubsection*{Centroid Loci}
In \cite{sergei2016-com} it was proved that the locus of the vertex and area centroids $C_0,C_2$ are conics over any Poncelet family whereas the locus of $C_1$ is not always a conic. Arseniy Akopyan reminded us that if a polygon circumscribes a circle (of center $O$), $C1,C_2,O$ are collinear and  $(C_1-O)=(3/2)(C_2-O)$ \cite{akopyan2022-private}. Referring to \cref{fig:perimeter}:

\begin{corollary}
Over the Steiner-Soddy porism of any number $N$ of sides, the locus of the perimeter centroid $C_1$ is a conic.
\end{corollary}

By symmetry, it follows that the major axes of the loci of $C_0,C_1$ are coaxial with the outer ellipse of the porism.

\begin{figure}
    \centering
    \includegraphics[trim=0 125 20 25,clip,width=.7\textwidth,frame]{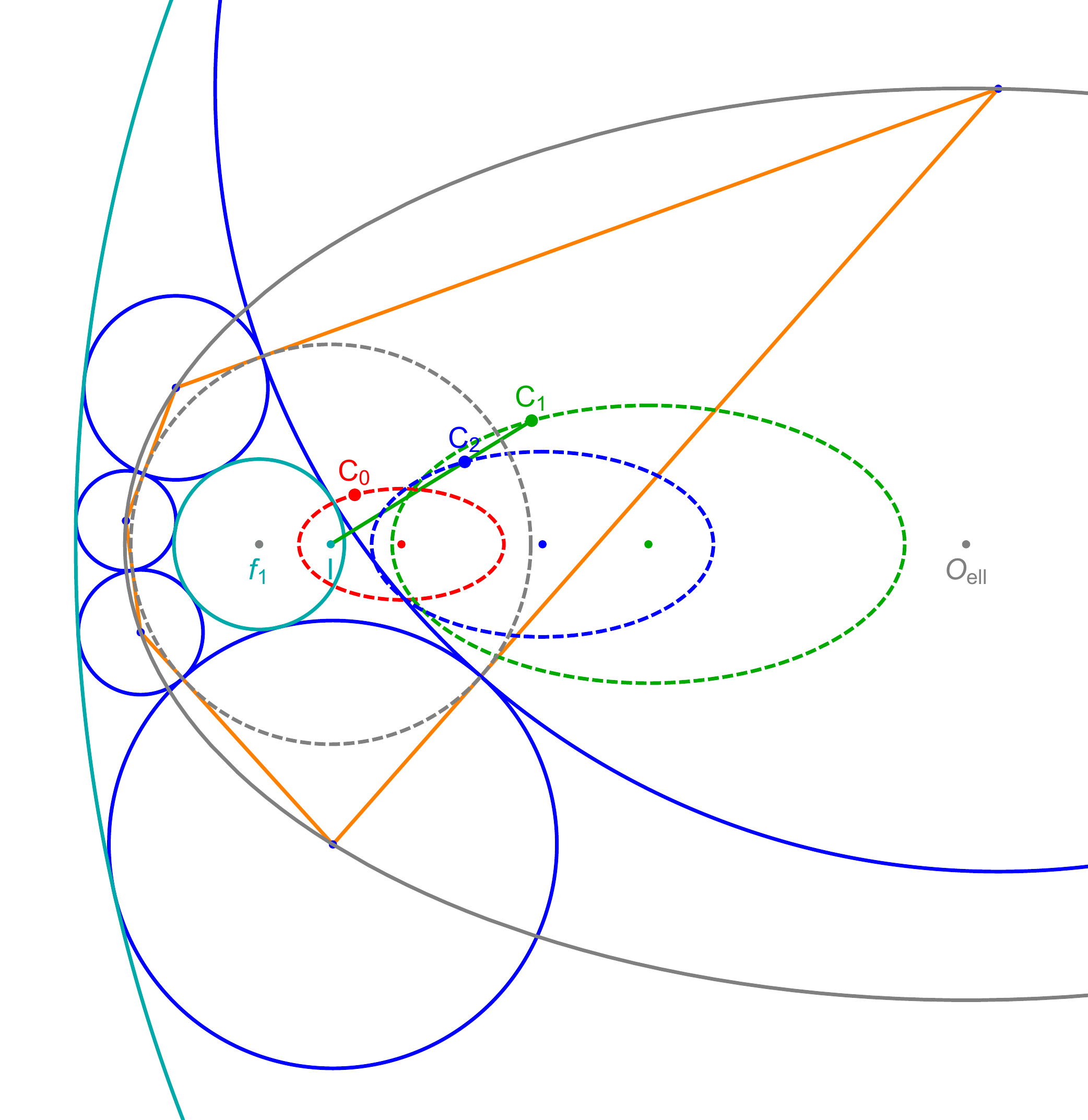}
    \caption{Over the Soddy-Steiner porism (orange), the loci of the vertex, area, and perimeter centroids $C_0,C_1,C_2$ are conics (dashed red, green, blue) coaxial with the outer ellipse (gray). Notice $I,C_2,C_1$ are collinear (green line) and $|IC_1|/|IC_2|=3/2$ \cite{akopyan2022-private}.}
    \label{fig:perimeter}
\end{figure}

\section{The Special case of N=3}
\label{sec:n3}
Consider an $N=3$ Steiner-Soddy family. We present a direct  proof  of \cref{thm:sum_tang_n} based on Descartes' theorem.

\begin{proof}
Let $1/\rho=1/r_1+1/r_2+1/r_3$. i.e., the sum of the curvatures of the (three) circles in the chain. By Descartes' theorem \cite{mw}:
{\small
\begin{align*}
\left(\frac{1}{r_1}+ \frac{1}{r_2}+ \frac{1}{r_3}+  \frac{1}{r_4}\right)^2 &= 
2\left( \frac{1}{r_1^2}+ \frac{1}{r_2^2}+ \frac{1}{r_3^2}+  \frac{1}{r_4^2}\right)\\
\left(\frac{1}{r_1}+ \frac{1}{r_2}+ \frac{1}{r_3}+  \frac{1}{r_5}\right)^2  &=
2 \bigg(\frac{1}{r_1^2}+ \frac{1}{r_2^2}+ \frac{1}{r_3^2}+  \frac{1}{r_5^2}\bigg)
\end{align*}
}
where $r_4$ and $r_5$ denotes the radii of the inner and outer Soddy circles. Subtracting and factoring, we obtain:
{\small
\[ \left(\frac{1}{r_4}-\frac{1}{r_5}\right)
\left(\frac{2}{\rho}+  \frac{1}{r_4}+\frac{1}{r_5}   \right)
=2 \left(\frac{1}{r_4}-\frac{1}{r_5}\right)\left(\frac{1}{r_4}+\frac{1}{r_5}\right)\]}
Hence $1/\rho=(1/r_4+1/r_5)/2$. Referring to \cref{fig:half-tangs-parabola}(left), it can be seen that:
\begin{equation}\tan\frac{A}{2}+\tan \frac{B}{2}+\tan\frac{C}{2}= \frac{r}{\rho}
\label{eq:soma_tg}
\end{equation}
Finally, since $r_4$, $r_5$, and $r$ are fixed, $\rho$ is constant, and so is  the sum of half tangents.
\end{proof}
\begin{remark}
The outer Soddy circle is a line when $r_5=+\infty$. By \cref{eq:soma_tg} this gives $\tan{\frac{A}{2}}+ \tan{\frac{B}{2}}+  \tan{\frac{C}{2}}=2$.
Note the latter was derived in \cite{yff2007-isoperim} in the context of showing if a certain condition in a triangle is satisfiable, which turns out to be equivalent to ours.
\end{remark}

\begin{proposition}
Let $\triangle{ABC}$ be in the Steiner-Soddy family and $\triangle{A'B'C'}$ be its $I$-pedal. Over the porism, the left and right hand side of each equation are conserved and identical:
\[\tan{\frac{A}{2}}+
\tan{\frac{B}{2}}+\tan\frac{C}{2}=\cot(A')+ \cot(B') + \cot(C')\]
\[\tan^2{\frac{A}{2}}+
\tan^2{\frac{B}{2}}+\tan^2\frac{C}{2}=
\cot^2(A')+ \cot^2(B') + \cot^2(C')\]

\end{proposition}
\begin{proof}
Referring to \cref{fig:half-tangs-parabola}(left), let the circles in the Steiner chain be centered at $A,B,C$, the vertices of $\P$. The perpendiculars in $A',B',C'$ meet at $O$, the incenter of $\triangle{ABC}$ as well as the circumcenter of $\triangle{A' B' C'}$. From properties of tangents, $AO\perp BC$. On the other hand  
$OC'\perp AB$. This proves that $\widehat{OAC'}$
and $\widehat{OC' B'}$ are either equal, or supplementary. Since both angles are acute, the latter is impossible. Hence 
$\widehat{OAC'}=\widehat{OC' B'}$. 
Similarly, $\widehat{O B C'}=\widehat{O C' A'}$. Thus:
\[\cot (C')=\cot(\frac{\widehat{A}+\widehat{B}}{2})=\cot(\frac{\pi -\widehat{C}}{2})=\tan\frac{C}{2}\]
\cref{eq:half-tan} finishes the proof.
\end{proof}

\begin{proposition}
In the $N=3$ case, the outer Soddy circle degenerates to a line when any one of the following equivalent conditions is fulfilled: (i) $\tau=2$; 
(ii) the outer conic is a parabola;
(iii) the circumcenter of $\B_3$ is at a co-vertex of the Brocard inellipse, whose aspect ratio is $\sqrt{5}/2=\phi-1/2 \approx 1.118$, where $\phi=(1+\sqrt{5})/2$ is the golden ratio. 
\end{proposition}

\begin{figure}
     \centering
     \begin{subfigure}[b]{0.48\textwidth}
         \centering
         \includegraphics[trim=1400 500 1000 100,clip,   width=\textwidth]{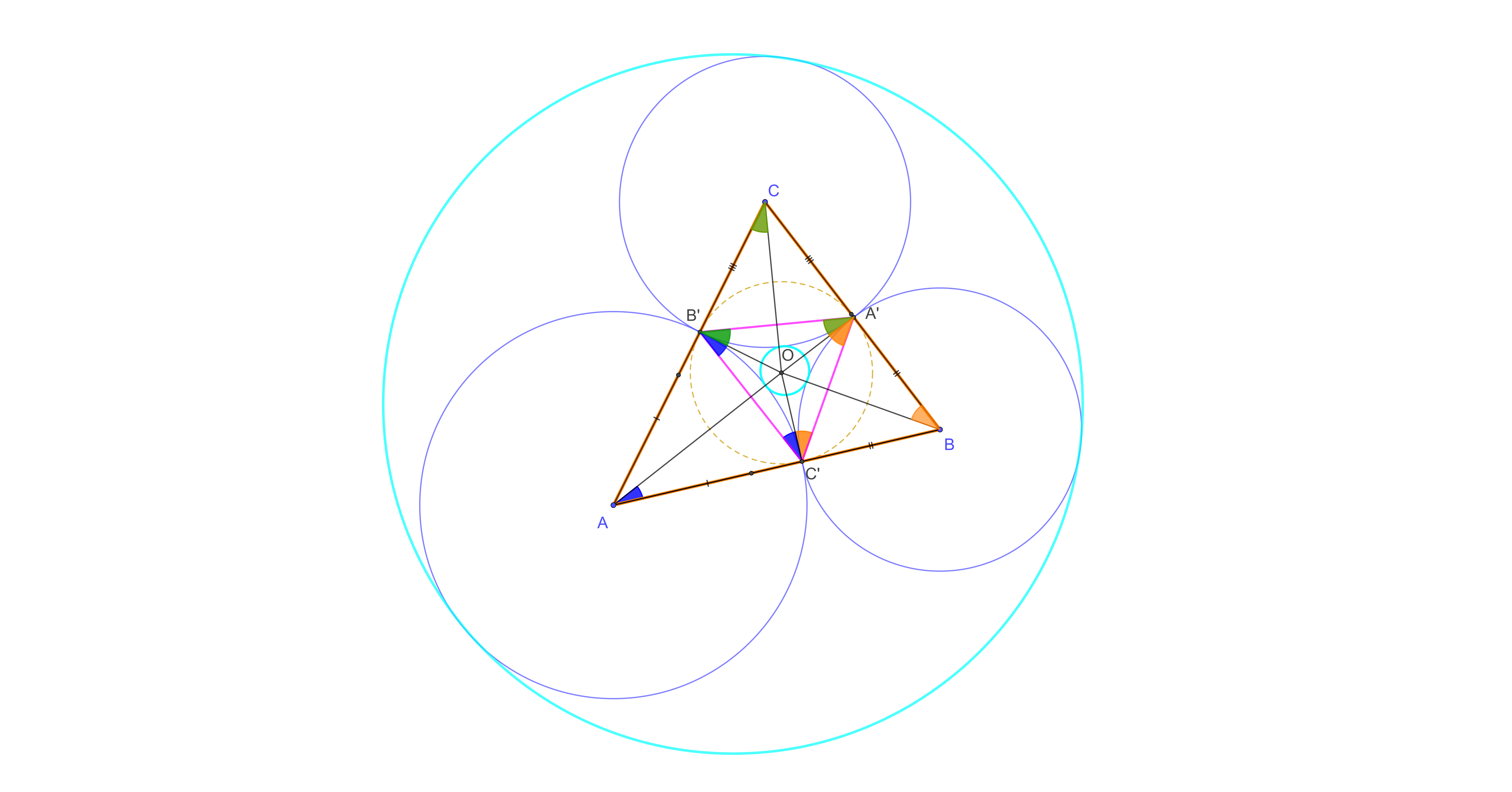}
     \end{subfigure}
     \rulesep
     \begin{subfigure}[b]{0.5\textwidth}
         \centering
        \includegraphics[width=\textwidth]{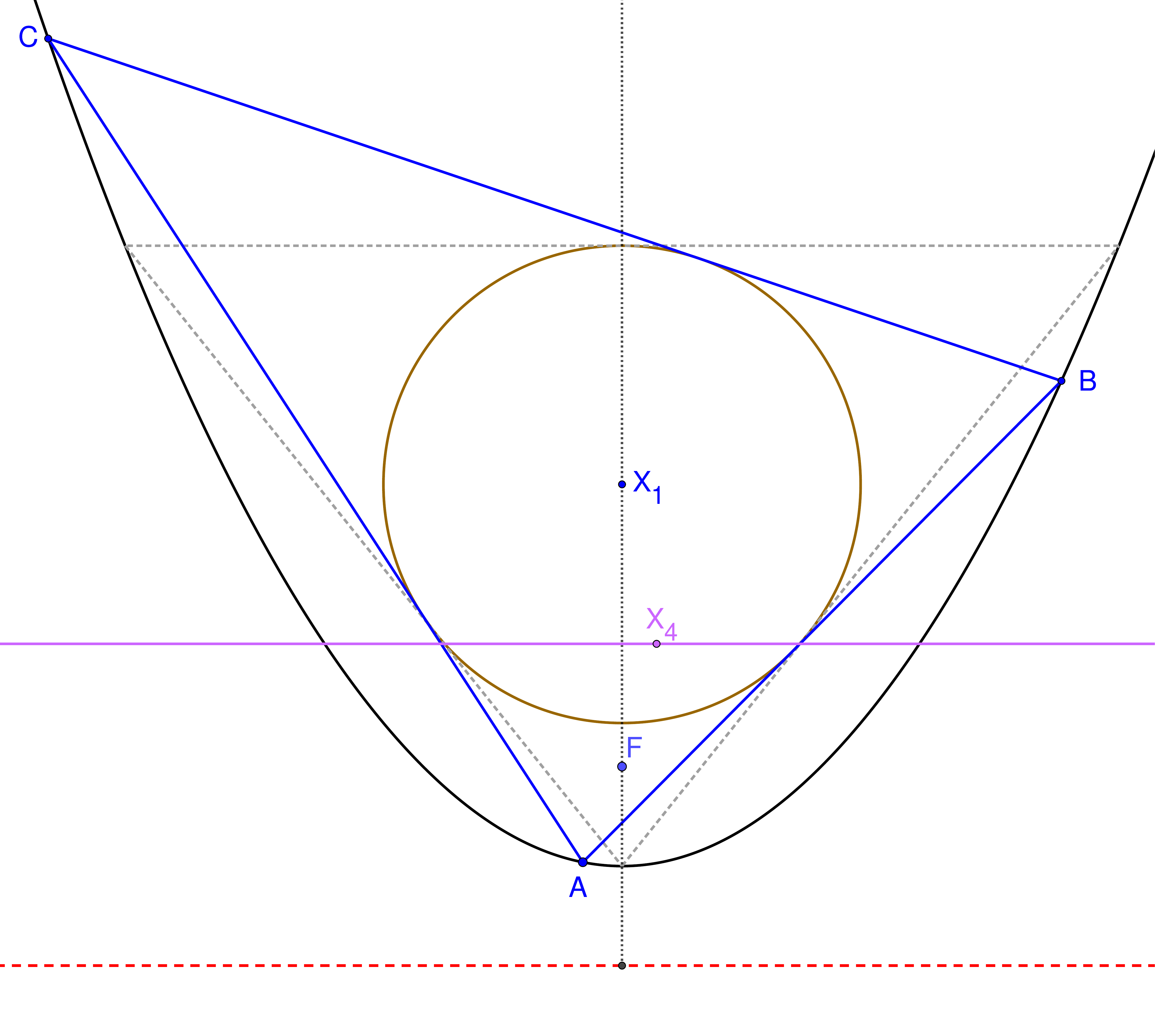}
     \end{subfigure}
     \caption{
     \textbf{Left:} Consider a $\triangle ABC$ and 3 mutually-tangent circles (blue) centered on $A,B,C$. These are tangent at the contact points $A',B',C'$ of the incircle with the sides of the triangle, which are the vertices of the pedal triangle with respect to the incenter $I$. Same-color angles  have the same measure. \textbf{Right:} Let $ABC$ be a parabola-inscribed Poncelet family of triangles whose caustic is a circle centered on the parabola axis. Over the family, the locus of the orthocenter $X_4$ is a line (purple) parallel to the directrix (dashed red).}
     \label{fig:half-tangs-parabola}
\end{figure}



\section{Loci in the N=3 case}
\label{sec:n3-loci}
Referring to \cref{fig:half-tangs-parabola}(right):

\begin{proposition}
For any parabola-inscribed Poncelet family whose caustic is an axis-centered circle, the locus of the orthocenter $X_4$ is a line parallel to the directrix. 
\end{proposition}

The following proof was kindly contributed by Alexey Zaslavsky \cite{zaslavsky2021-private}:

\begin{proof}
Consider the unit parabola $y=x^2$, and let $y_0$, $r$ be the center and the radius of the incircle. Then the points $(r,y_0-r)$ and $(-r,y_0-r)$ lie on the parabola, thus $y_0=r^2+r$. Now let $A(a,a^2)$, $B(b,b^2)$, $C(c,c^2)$ be the vertices of the triangle. Then the distances from the incenter $(0,r^2+r)$ to lines AB and AC equal r and we obtain b+c and bc as functions of a. After this we have that the ordinate of orthocenter $-(1+ab+ac+bc)=r^2+2r-1$ do not depend on $a$.
\end{proof}

\begin{remark} Consider the parabola $4cy=x^2$ and the circle centered at $(0,y_0)$. For the pair to admit Poncelet triangles, it can be shown $r= 4y_0c/\sqrt{16c^2 + x_0^2}$, where $x_0=2\sqrt{-2c^2 + y_0c + 2 \sqrt{c^3 (c + y0}}$. In this case, the locus of $X_4$ is the line $y=-4c + x_0^2/(4c)=(-6c^2 + y_0c + 2\sqrt{c^3(c + y_0}))/c$.
\end{remark}
 
\begin{corollary}
Over a parabola-inscribed Soddy-Steiner family, the locus of its orthocenter $X_4$ is a line parallel to the directrix.
\end{corollary}

For example, assume the parabola is $4cy=x^2$. Then the locus of $X_4$ is the line $y=-7c/4$.

\begin{figure}
    \centering
    \includegraphics[trim=20 300 0 0,clip,width=.8\textwidth,frame]{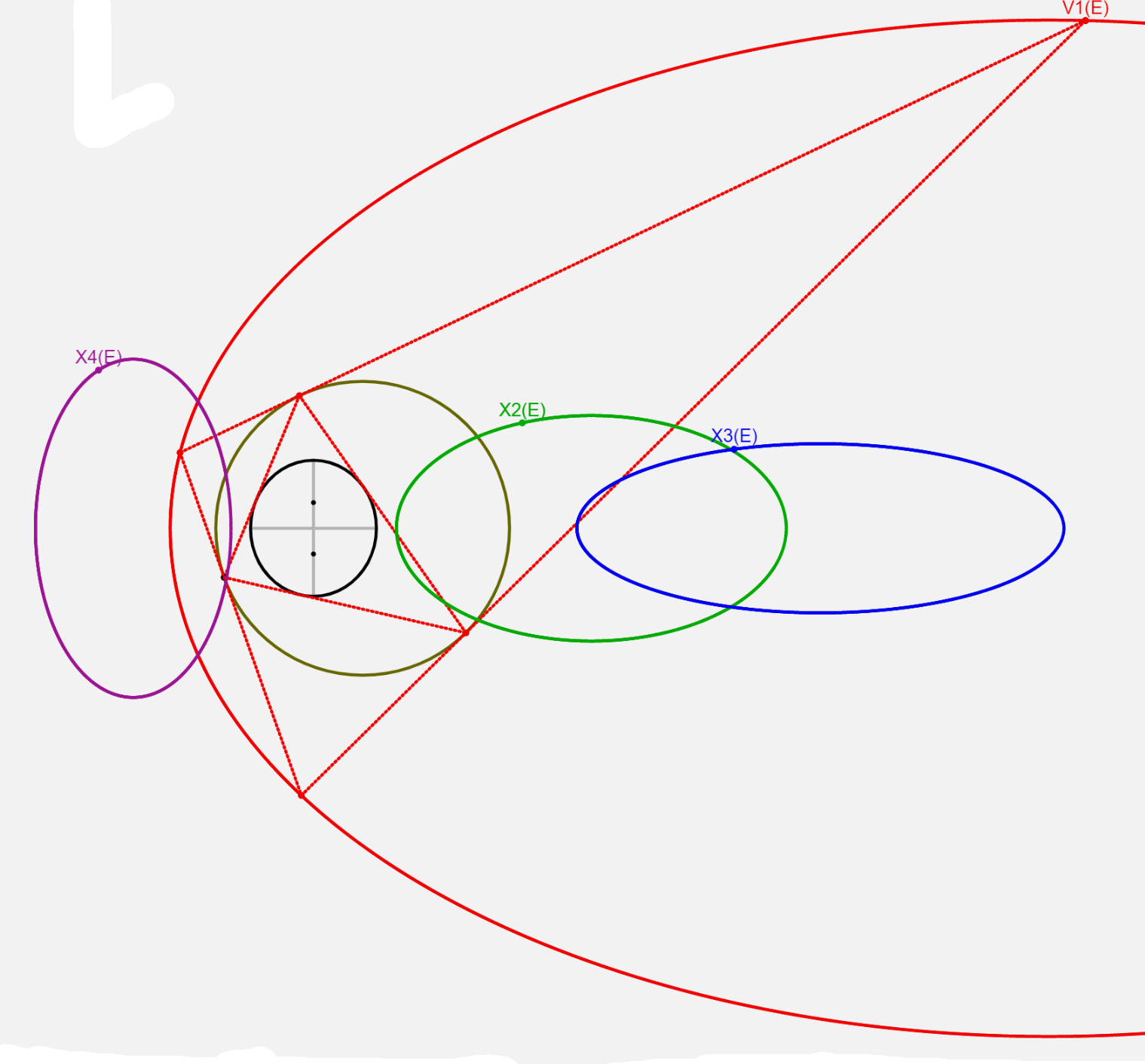}
    \caption{Over the Steiner-Soddy family, the loci of the centroid $X_2$ (green), circumcenter $X_3$ (blue) and orthocenter $X_4$ (purple) are conics.}
    \label{fig:x234}
\end{figure}
 
 It is known that if the intouch triangle is acute, its circumcenter $X_3$ (resp. symmedian $X_6$) coincides with the incenter $X_1$ (resp. Gergonne point $X_7$) of the reference.

While the intouch of $\S_3$ is acute (its harmonic $O$-pedal), $X_1$ and $X_7$ are stationary. Otherwise, while $S_3$ is everted (it is inscribed in a hyperbola), $X_1$ and $X_7$ follow arcs curves of degree at least four. Here are some experimental observations about the loci of $X_k$, some of which are explicitly derived in \cref{app:explicit}:

\begin{compactenum}
\item $X_k, k=$2, 3, 4, 5 are conics, as predicted in \cite{helman2021-theory}, see \cref{fig:x234}.
\item $X_k, k=$6, 8, 9, 10 are also conics, though not in general \cite{helman2021-theory}.
\item $X_k, k=$13, 14, 15, 16, 80 are circles.
\item $X_k, k=$20, 77, 170 are segments along the major axis of $\E$.
\item $X_{105}$ is a circle centered on the major axis of $\E$ and tangent to it at two points.
\end{compactenum}

\section{Comparing Poncelet Conservations}
\label{sec:cons}
Conservations of various Poncelet families studied so far as well as their polar images with respect to certain significant points is shown in \cref{tab:conservations}. As we study more of these canonical families, one of our goals (still elusive) is to identify a functional pattern in the conserved quantities.

{\small
\begin{table}
\centering
\begin{tabular}{|l|l|l|l|l|}
\hline
family & invariant & polar center/conic & new family & new invariant \\
\hline
 \multirow{2}{*}{confocal} & \multirow{2}{*}{$\sum\cos\theta_i$ \cite{akopyan2020-invariants,bialy2020-invariants}} & outer ell & tangential & $\sum\cos(2\theta_i)$ \cite{reznik2021-fifty-invariants} \\
& & inner/outer focus & bicentric  & $\sum\cos\theta_i$ \\
\hline
\multirow{2}{*}{bicentric} & \multirow{2}{*}{$\sum\cos\theta_i$ \cite{roitman2021-bicentric}} & outer circle & bic-tang & $\sum\sin(\theta_i/2)$ \cite{bellio2021-parabola-inscribed} \\
& & incircle inward & MacBeath & $\prod\cos\theta_i$ \\
 \hline
\multirow{2}{*}{incircle} & \multirow{2}{*}{$\sum\cos\theta_i$ \cite{akopyan2020-invariants}} & outer ell & circum & $\prod\sin\theta_i$ \\
& & incircle inward & circum & $\prod\sin\theta_i$ \\
\hline
{circum} & $\prod\cos\theta_i$ \cite{akopyan2020-invariants,caliz2020-area-product} & 
outer circle & incircle & $\sum\cos\theta_i$ \\ 
\hline
\multirow{2}{*}{homothetic} & \multirow{2}{*}{$\sum\cot\theta_i$ \cite{galkin2022-regular}}& outer focus & harmonic & $\sum\cot\theta_i$ \\
& & inner focus & Steiner-Soddy &$\sum\tan(\theta_i/2)$  \\
\hline
\multirow{2}{*}{harmonic} & \multirow{2}{*}{$\sum\cot\theta_i$ \cite{roitman2021-harmonic}} & outer circle & Steiner-Soddy & $\sum\tan(\theta_i/2)$\\
& & focus & harmonic & $\sum{\cot\theta_i}$ \\
\hline
\end{tabular}
\caption{Poncelet families, polar transformations, and invariants}
\label{tab:conservations}
\end{table}
}

\subsubsection*{Acknowledgements}
\noindent We thank Alexey Zaslavsky and Arseniy Akopyan for contributing some proofs and facts, and Richard Schwartz for useful discussions. The first author is fellow of CNPq and coordinator of Project PRONEX/ CNPq/ FAPEG 2017 10 26 7000 508.

\appendix

\section*{Appendix: Explicit Formulas}
\label{app:explicit}
Referring to \cref{fig:intro-n5}, let the inversion circle be centered at $\mbox{inv}=(x_0,0)$ and have radius $\lambda$. and $\alpha=\pi/N$. Let $\E$ denote the conic the Steiner-Soddy family is inscribed to and $\C$ its circular caustic.

\begin{proposition}
The $x$ coordinates of foci $f_1$ and $f_2$ and vertex $V$ of $\mathcal{E}$ are given by:
{\scriptsize
\begin{align*}
f_{1,2}=&\frac{ x_0\left(R^4\cos^4\alpha  - R^2 \cos^2\alpha ({\lambda}^2 - 2x_0^2) \pm 2 R^2 \sin\alpha {\lambda}^2 + 2R^2({\lambda}^2 - 2x_0^2) - x_0^2({\lambda}^2 - x_0^2) \right)  } {     R^4\cos^4\alpha+ 2R^2\cos^2\alpha x_0^2 - 4R^2x_0^2 + x_0^4 }\\
V_x=&\frac{ \cos(2\alpha)R^2 x_0 + R(R x_0 + 2 {\lambda}^2 - 4x_0^2) - 2x_0({\lambda}^2 - x_0^2) }{ R^2\cos(2\alpha) + R^2 - 4R x_0 + 2x_0^2 }
\end{align*}
}
\end{proposition}

\begin{proposition}
The parameters of the caustic $\C=(I,r)$ are given by:
{\small
\[  I=\left[x_0 + \frac{2x_0\lambda^2}{ 2(R^2\cos^2\alpha -  x_0^2)},0\right],\;\;\;r=\frac{\lambda^2 R\cos\alpha}{ |R^2\cos^2\alpha - x_0^2|}\]}
\end{proposition}

\begin{proposition}
Let Brocard inellipse (caustic of the pedal family with respect to $I$) be centered at $O'$ with semi-axes $a',b'$. These are given by:
{\scriptsize
\begin{align*}
    O'&=\left[x_0 +\frac{ \lambda^2 x_0(R^2\cos ^2\alpha \cos(2\alpha) - x_0^2)}{ R^2(R^2 - 4x_0^2)\cos ^4\alpha + 2R^2x_02\cos ^2\alpha  + x_0^4},0\right]\\
    a'&=\frac{\lambda^2 R(x_0^2-R^2\cos^2\alpha     )\cos^2\alpha }{R^2(R^2- 4  x_0^2) \cos ^4 \alpha   + 2 R^2 x_0^2\cos ^2\alpha   + x_0^4}\\
    b'&=\frac{\lambda ^2R\cos ^2\alpha }{\sqrt{R^2(R^2- 4  x_0^2) \cos ^4 \alpha + 2 R^2x_0^2\cos ^2\alpha  + x_0^4}}\\
\end{align*}
}
\end{proposition}
\begin{proposition} The loci of $X_2$, $X_3$ and $X_4$ are the conics given by: {\scriptsize 
\begin{align*}
   X_2: & (R^2 - 4 x_0^2) ( R^4 - 56R^2x_0^2 + 16x_0^4) x^2 + (R^2 - 4 x_0^2)^3   y^2\\
    &- 2 x_0 (R^6 - 60 R^4 x_0^2 + 12 R^4 \lambda^2 + 240 R^2 x_0^4 - 192 R^2 x_0^2 \lambda^2 - 64 x_0^6 + 64 x_0^4 \lambda^2) x \\
   &+   x_0^2 ( R^6 - 60 R^4 x_0^2 + 24 R^4 \lambda^2 + 240 R^2  x_0^4 - 384 R^2  x_0^2 \lambda^2 + 144 R^2 \lambda^4 - 64  x_0^6 + 128  x_0^4 \lambda^2 - 64  x_0^2 \lambda^4)=0 \\
   X_3: & ( R^4 - 56R^2x_0^2 + 16x_0^4) (R^2 - 4x_0^2)^4 x^2 + (R^2 - 4x_0^2)^2(R^4 + 40 R^2  x_0^2 + 16  x_0^4)^2 y^2 \\
    & -2 x_0 (R^2 - 4x_0^2) (R^{10} - 68 R^8  x_0^2 + 28 R^8 \lambda^2 + 736 R^6  x_0^4 - 928 R^6  x_0^2 \lambda^2 - 2944 R^4  x_0^6 - 768 R^4  x_0^4 \lambda^2 \\
    &+ 4352 R^2  x_0^8 - 2560 R^2  x_0^6 \lambda^2 - 1024  x_0^{10} + 1024  x_0^8 \lambda^2) x \\
    &+   x_0^2(R^{12} + (-72  x_0^2 + 56 \lambda^2) R^{10} + (1008  x_0^4 - 2080  x_0^2 \lambda^2 + 784 \lambda^4) R^8 - 256  x_0^2 (23  x_0^4 - 23  x_0^2 \lambda^2 \\
    &+ 16 \lambda^4) R^6  
    + 256  x_0^4 (63  x_0^4 + 4  x_0^2 \lambda^2 - 18 \lambda^4) R^4 - 2048  x_0^6 (x_0^2 - \lambda^2) (9  x_0^2 - 2 \lambda^2) R^2 + 4096  x_0^8 (x_0^2 - \lambda^2)^2 ) =0\\
    X_4: & (R^2 - 4x_0^2)^4 x^2 +   (R^4 - 56R^2x_0^2 + 16x_0^4)(R^2 - 4x_0^2)^2 y^2\\
   &x_0(-2 R^8 + 32 R^6 x_0^2+ 40 R^6   \lambda^2 - 192 R^4 x_0^4 + 96 R^4 x_0^2 \lambda^2 + 512 R^2 x_0^6 - 1152 R^2 x_0^4 \lambda^2 - 512 x_0^8 + 512 x_0^6 \lambda^2) x \\
   &+x_0^2(R^8 - 16 R^6 x_0^2 - 40 R^6 \lambda^2 + 96 R^4 x_0^4 - 96 R^4 x_0^2 \lambda^2 + 400 R^4 \lambda^4 - 256 R^2 x_0^6 + 1152 R^2 x_0^4 \lambda^2\\
  & - 896 R^2 x_0^2 \lambda^4  
     + 256 x_0^8 - 512 x_0^6 \lambda^2 + 256 x_0^4 \lambda^4)=0
\end{align*}
  }

\end{proposition}
\begin{proposition} The locus of $X_6$  is the ellipse given by:
{\scriptsize 
\begin{align*}
X_6:&  (R^2 - 4 x_0^2)  (R^8 + 544 R^4 x_0^4 + 256 x_0^8)  (R^2 + 4  x_0^2)^2  x^2+ (R^4 + 16  x_0^4)^2  (R^2 - 4  x_0^2)^3  y^2\\
&- 2  x_0  (R^2 + 4  x_0^2)^2  (R^{10} - 4  R^8    \lambda^2 - 4  R^8  x_0^2 - 160  R^6    \lambda^2  x_0^2 + 544  R^6  x_0^4 \\
&+ 1536  R^4    \lambda^2  x_0^4 - 2176  R^4  x_0^6  
 - 512  R^2    \lambda^2  x_0^6 + 256  R^2  x_0^8 + 1024    \lambda^2  x_0^8 - 1024  x_0^{10})  x\\
&+ x_0^2  (R^2 + 4  x_0^2)^2  ( R^{10}   ( 8  \lambda^2 + 4 x_0^2) R^8 + (16  \lambda^4 - 320  \lambda^2 x_0^2 + 544 x_0^4) R^6 - 64 x_0^2 (15  \lambda^4 \\
&- 48  \lambda^2 x_0^2 
 + 34 x_0^4) R^4 + 256 x_0^4 ( \lambda^2 - x_0^2)   (3  \lambda^2 - x_0^2) R^2 - 1024 x_0^6 ( \lambda^2 - x_0^2)^2 )=0
\end{align*}
  }

\end{proposition}

\begin{proposition} When the center of inversion is internal to the Soddy circle, $S_{reg}$ the locus $X_{15}$ is the circle given by:
{\scriptsize 
\begin{align*}
    X_{15}:&\left(x - \frac{x_0(-256 x_0^6 + 144 R^2 x_0^4 -24 ( R^4 + 6 R^2 \lambda^2) x_0^2 + R^6 + 12 R^4 \lambda^2)  }{R^6 - 24 R^4 x_0^2 + 144 R^2 x_0^4 - 256 x_0^6}\right)^2 + y^2\\
    &=\frac{ 36864 R^2 x_0^8 \lambda^4}{(R^6 - 24 R^4 x_0^2 + 144 R^2 x_0^4 - 256 x_0^6)^2}
\end{align*}
}
\end{proposition}
\begin{proposition}
The locus of $X_{20}$ is a segment  $[X^-_{20},X^+_{20}]$ contained in axis $x$ of length $L_{20}$. These are given by:
{\scriptsize 
\begin{align*}
    X^+_{20} &= \frac{x_0\left(\zeta + 4 x_0 R^5 + 304 R^3 x_0 \lambda^2 - 64 x_0^3 (x_0^2  - \lambda^2) ( R + x_0) \right)}{ (R^2 + 8 R x_0 + 4 x_0^2) (R - 2 x_0)^3 (R + 2 x_0)} \\
X^-_{20}&= \frac{x_0(\zeta - 4 x_0 R^5 - 304 R^3 x_0 \lambda^2 + 64 x_0^3 (x_0^2  - \lambda^2) (R - x_0) )  }{(R^2 - 8 R x_0 + 4 x_0^2) (R + 2 x_0)^3 (R - 2 x_0)}\\
L_{20}&=\frac{18432  x_0^4  R^5  \lambda^2}{(R^2 - 4  x_0^2)^3  (R^4 - 56  R^2  x_0^2 + 16  x_0^4)}
\end{align*}
where $\zeta=R^6 + (76 \lambda^2 - 28 x_0^2) R^4 +16 x_0^2 (12 \lambda^2 + 7 x_0^2) R^2$.
}

\end{proposition}

\bibliographystyle{splncs04}
\bibliography{999_refs,999_refs_rgk}

\end{document}